\documentclass[a4paper,12pt]{article}
\pdfoutput=1
\usepackage{amssymb, amsmath, bbm, amsthm}
\usepackage{fullpage}
\usepackage{url}
\usepackage{hyperref}
\usepackage{comment}
\usepackage{graphicx}
\theoremstyle{plain}
\newtheorem{theorem}{Theorem}[section]
\newtheorem{lemma}[theorem]{Lemma}
\newtheorem{corollary}[theorem]{Corollary}
\newtheorem{proposition}[theorem]{Proposition}
\newtheorem{example}[theorem]{Example}

\numberwithin{equation}{section}

\theoremstyle{definition}
\newtheorem{definition}[theorem]{Definition}

\theoremstyle{remark}
\newtheorem{remark}[theorem]{Remark}

\newcommand{\R}{{\mathbb R}}

\newcommand{\N}{{\mathbb N}}

\newcommand{\F}{{\mathbb F}}

\newcommand{\Prob}{\mathbb{P}}

\newcommand{\E}{\mathbb{E}}

\newcommand{\Indi}{\mathop{\mathbbm{1}}\nolimits}

\newcommand{\mmu}{\mathop {\bar{x}_{\mu}}}

\newcommand{\supp}{\mbox{supp}}

\usepackage{color}

\usepackage{asymptote}

\begin{asydef}
usepackage("amsmath");
texpreamble("\newcommand{\mmu}{\mathop {\bar{x}_{\mu}}}");
\end{asydef}

\title{From minimal embeddings to minimal diffusions}

\author{Alexander
  M.G. Cox\thanks{Department of Mathematical Sciences, University of
    Bath, U.K. \href{mailto:a.m.g.cox@bath.ac.uk}{\nolinkurl{a.m.g.cox@bath.ac.uk}}}
\and
  Martin
  Klimmek\thanks{Nomura Centre for Mathematical Finance, Mathematical
 Institute, University of Oxford, U.K. \href{mailto:martin.klimmek@maths.ox.ac.uk}{\nolinkurl{martin.klimmek@maths.ox.ac.uk}}}}

\date{\today}
\begin{document}

\maketitle

\begin{abstract}
  There is a natural connection between the class of
  diffusions, and a certain class of solutions to the Skorokhod
  Embedding Problem (SEP). We show that the important concept
  of minimality in the SEP leads to the new and useful concept of a {\it   
  minimal diffusion}. Minimality is closely related to the martingale 	
  property. A diffusion is minimal if it minimises the expected local time at
  every point among all diffusions with a given distribution at an
  exponential time. Our approach makes explicit the connection between
  the boundary behaviour, the martingale property and the local time
  characteristics of time-homogeneous diffusions.

\end{abstract}

\section{Introduction}
This article shows that there is a one-to-one correspondence between a
(generalised) diffusion in natural scale and a class of solutions to
the Skorokhod Embedding Problem (SEP) for a Brownian motion. The main
contribution of this article is to establish the concept of a {\it
  minimal diffusion}, which is motivated by the correspondence.

The link between embedding and diffusion involves an additive
functional which can be used to define both a diffusion (via a time
change) and a stopping time (the first time that the additive
functional exceeds an independent exponentially distributed random
variable). The fundamental connection was first observed by Cox et.\
al.\ in \cite{CoxHobsonObloj:2011}. In that article, the authors
proved existence of a (martingale) diffusion with a given law at an
exponentially distributed time by exploiting the link between
martingale diffusions and minimal embeddings. It was then demonstrated
in \cite{MK:12} that a diffusion's speed measure can be represented in
terms of its exponential time law without recourse to embedding
theory.

The general theme of this paper is that we can exploit the
correspondence between diffusions and the SEP to introduce notions
from the theory of SEP into the diffusion setting. This leads us to
introduce the novel concept of a {\it minimal diffusion}, which we
argue is the canonical diffusion in the family of diffusions with
given law at a random time, and corresponds to the minimal embedding.




In the literature on the SEP, the concept of a minimal stopping time
is crucial. The problem of constructing stopping times so that a
stopped process has a given law degenerates unless there is a notion
of what constitutes a `good' solution. An early definition of `good'
was `finite expectation'. The concept of a `minimal solution' was
first introduced by Monroe, \cite{Monroe:72}. Loosely speaking, a
stopping time is minimal if there is no almost surely smaller stopping
time for the process which results in the same distribution. The
property of minimality connects much of the modern literature on
Skorokhod embeddings and their applications in mathematical
finance. The main contribution of this article is to show that the
concept of minimality extends naturally as a property of
diffusions. In leading up to this result, we provide a novel
characterisation of minimal stopping times in terms of the local
times, which is important in transferring the concept of minimality
from solutions of the SEP to diffusions.


From the point of view of classical diffusion theory, the article provides 
insight into the additive functional that defines a diffusion and introduces
the intuitive concept of `minimality' to diffusion theory. We argue that 
the property of minimality is in many ways more natural than the 
martingale property. The concept of minimality dis-entangles 
distributional properties of the process from its boundary behaviour. There 
may not exist a martingale diffusion with a given exponential time law, but
there is always a minimal diffusion.

Our study of this problem was motivated in part by the observation
that the construction of time-homogenous diffusions via time-change
was closely related to the solution to the SEP given by Bertoin \& Le
Jan \cite{BLJ}. Indeed, one of the proofs in
\cite{CoxHobsonObloj:2011} relied on explicitly constructing a process
which enabled the use of the results of \cite{BLJ}. The Bertoin Le-Jan
(BLJ) embedding stops a process at the first time an additive
functional grows larger than the accumulated local time at the
starting point. However many of the same quantities appear in the
construction of both the BLJ embedding and the time-homogenous
diffusion. By considering the BLJ embedding, we are able to explain
why this is the case, and thereby provide new insight into the BLJ
construction, and its connection to time-homogenous diffusions.


\section{Preliminaries}

\subsection{Generalised Diffusions}
Let $m$ be a non-negative, non-zero Borel measure on an interval $I \subseteq \R$, with left endpoint $a$ and right endpoint $b$ (either or both of which may be infinite). Let $x_0 \in (a,b)$ and let $B=(B_t)_{t \geq 0}$ be a Brownian motion started at $B_0=x_0$ supported on a filtration $\F^B=({\mathcal F}_u^B)_{u\geq 0}$ with local time process $\{ L_u^x ; u \geq 0, x \in \R \}$. Define $\Gamma$ to be the continuous, increasing, additive functional
\[\Gamma_u = \int_{\R} L_u^x m(dx),\]
and define its right-continuous inverse by
\[A_t = \inf \{u \ge 0 : \Gamma_u > t \}. \] If $X_t = B_{A_t}$ then $X=(X_t)_{t \geq 0}$ is a one-dimensional generalised diffusion in natural scale with $X_0=x_0$ and speed measure $m$. Moreover, $X_t \in I$ almost surely for all $t \geq 0$. In contrast to diffusions, which are continuous by definition, generalised diffusions may be discontinuous if the speed measure places no mass on an interval. For instance, if the speed measure is purely atomic, then the process is a birth-death process in the sense of Feller \cite{feller}. See also Kotani and Watanabe \cite{kotaniwatanabe}. In the sequel, we will use diffusion to denote the class of generalised diffusions, rather than continuous diffusions.

Let $H_x=\inf\{u:X_u=x\}$. Then for $\lambda>0$ (see e.g. \cite{Salminen}),
\begin{equation} \label{eq:eigenfunction}
\E_{x}[e^{-\lambda H_y}]= \left\{\begin{array}{ll}
\frac{\varphi_\lambda(x)}{\varphi_\lambda(y)}  &\; x \leq y \\
\frac{\phi_\lambda(x)}{\phi_\lambda(y)} &\; x \geq y ,
\end{array}\right.
\end{equation}
where $\varphi_\lambda$ and $\phi_\lambda$ are respectively a strictly increasing and a strictly decreasing solution to the differential equation
\begin{equation} \label{eq:differentialsc}
\frac{1}{2} \frac{d^2}{dm dx} f = \lambda f.
\end{equation}
The two solutions are linearly independent with Wronskian 
$W_\lambda=\varphi_\lambda' \phi_\lambda-\phi_\lambda' \varphi_\lambda$, which is a positive constant. 

The solutions to (\ref{eq:differentialsc}) are called the $\lambda$-eigenfunctions of the diffusion. We will scale the $\lambda$-eigenfunctions so that $\varphi_\lambda(x_0)=\phi_\lambda(x_0)=1$. 

\subsection{The Skorokhod Embedding Problem}
We recall some important notions relating to the Skorokhod Embedding Problem (SEP). The SEP can be stated as follows: given a Brownian motion $(B_t)_{t \ge 0}$ (or, more generally, some stochastic process) and a measure $\mu$ on $\R$, a solution to the SEP is a stopping time $\tau$ such that $B_{\tau} \sim \mu$. We refer to Ob\l\'oj \cite{Obloj:04} for a comprehensive survey of the history of the Skorokhod Embedding Problem.

Many solutions to the problem are known, and it is common to require some additional assumption on the process: for example that the stopped process $(B_{t \wedge \tau})_{t \ge 0}$ is uniformly integrable. In the case where $B_0 = x_0$, this requires some additional regularity on $\mu$ --- specifically that $\mu$ is integrable, and $x_0 = \bar{x}_\mu = \int y \, \mu(dy)$. We recall a more general notion due to Monroe \cite{Monroe:72}:

\begin{definition} \label{def:minimal}
  A stopping time $\tau$ is {\it minimal} if, whenever $\sigma$ is
  another stopping time with $B_{\sigma} \sim B_{\tau}$ then $\sigma
  \le \tau \ \Prob$-a.s.{} implies $\sigma = \tau \ \Prob$-a.s..
\end{definition}

That is, a stopping time $\tau$ is minimal if there is no strictly
smaller stopping time which embeds the same distribution. It was
additionally shown by Monroe that if the necessary condition described
above was true ($\mu$ integrable with mean $x_0$) then minimality of
the embedding $\tau$ is equivalent to uniform integrability of the
stopped process. In the case where the means do not agree, minimality
of stopping times was investigated in Cox and Hobson
\cite{CoxHobson:06} and Cox \cite{Cox:08}. Most natural solutions to
the Skorokhod Embedding problem can be shown to be minimal. In some
cases, constructions can be extended to non-minimal stopping times ---
see for example Remark~4.5 in \cite{cox_roots_2013}.

To motivate some of our later results, we give the following alternative characterisation of minimality for stopping times of Brownian motion. We note that this condition was first introduced by Bertoin and Le Jan \cite{BLJ}, as a property of the BLJ embedding. We write $L_t^a$ for the local time at the level $a$. When $a=0$, we will often simply write $L_t$.

\begin{lemma} \label{lem:minimal}
  Let $\mu$ be an integrable measure, and suppose $\tau$ embeds $\mu$ in a Brownian motion $(B_{t})_{t \ge 0}$ with $B_0 = x_0$. Let $a \in \R$ be fixed. Then $\tau$ is minimal if and only if $\tau$ minimises $\E[L_{\sigma}^a]$ over all stopping times $\sigma$ embedding $\mu$.
\end{lemma}

\begin{proof}
  Without loss of generality, we may assume $a=0$. Let $\tau$ be an
  embedding of $\mu$. Note that $L_t-|B_t|$ is a local martingale. Let
  $\tau_N$ be a localising sequence of stopping times such that $\tau_N \uparrow \tau$ and $\{L_{t\wedge \tau_N}-|B_{t\wedge \tau_N}|\}$ is a family of martingales. We have 
  \begin{equation*}
    \E[L_{\tau_N}] = -|x_0| + \E[|B_{\tau_N}|]
  \end{equation*}
  and hence
  \begin{equation*}
    \lim_{N \to \infty}\E[L_{\tau_N}] = -|x_0| + \lim_{N \to \infty}\E[|B_{\tau_N}|].
  \end{equation*}
  Since $|x| = x + 2 x_-$ (where $x_- = \max\{0,-x\}$) we can write
  \begin{equation*}
     \lim_{N \to \infty}\E[|B_{\tau_N}|] = \lim_{N \to \infty} \E
     [B_{\tau_N}] + 2\lim_{N \to \infty} \E[ ( B_{\tau_N})_-].
  \end{equation*}
  Since $\tau_N$ is a localising sequence, $\lim_{N \to \infty} \E
  [B_{\tau_N}] = x_0$, while, by Fatou's Lemma, $\lim_{N \to \infty}
  \E [( B_{\tau_N})_-] \ge \E [(B_{\tau})_-]$, and the final term depends
  only on $\mu$. So $\E [ L_\tau] \ge x_0-|x_0| + 2 \E [(B_{\tau})_-]$.

  This is true for any embedding $\tau$, however in the case where
  $\tau$ is minimal, we observe that $\left\{(B_{\tau_N})_-\right\}_{N
    \in \N}$ is a UI family, by Theorem~5 of \cite{CoxHobson:06}, and
  so we have the equality: $\lim_{N \to \infty} \E [( B_{\tau_N})_-] =
  \E [(B_{\tau})_-]$, and hence $\E [ L_\tau] = x_0-|x_0| + 2 \E [(B_{\tau})_-]$.

  For the converse, again using Theorem~3 of \cite{CoxHobson:06}, we
  observe that it is sufficient to show that $\left\{(B_{t \wedge
      \tau})_-\right\}_{t \ge 0}$ is a UI family.

  Note that for any integrable measure $\mu$ a minimal embedding
  exists, and therefore any embedding $\tau$ which minimises
  $\E[L_{\tau}]$ over the class of embeddings must have $\lim_{t \to
    \infty} \E[( B_{\tau \wedge t})_-] = \E
  [(B_{\tau})_-]$. However, suppose for a contradiction that
  $\left\{(B_{t \wedge \tau})_-\right\}_{t \ge 0}$ is not a UI
  family. Since $(B_{t \wedge \tau})_- \to (B_{\tau})_-$ in
  probability as $t \to \infty$, this implies we cannot have
  convergence in $\mathcal{L}^1$ (otherwise the sequence would be
  UI). It follows that
  \begin{equation*}
    \E \left[ | (B_{t \wedge \tau})_- - (B_{\tau})_-|\right] \to \varepsilon >0,
  \end{equation*}
  as $t \to \infty$. But 
  \begin{equation*}
    \E \left[| (B_{t \wedge \tau})_- - (B_{\tau})_-| \right]\le \E \left[\left|\left(
        (B_{t \wedge \tau})_- - (B_{\tau})_-\right) \Indi_{\{\tau
        \le t\}}\right|\right] + \E [ (B_\tau)_- \Indi_{\{\tau
        > t\}}] + \E [ (B_t)_- \Indi_{\{\tau
        > t\}}].
  \end{equation*}
  It follows that $\lim_{t \to \infty}\E[(B_{t})_- \Indi_{\{\tau >
    t\}}] \ge \varepsilon$. But $\E[(B_{t \wedge \tau})_-
  \Indi_{\{\tau \le t\}}] \to \E[(B_{\tau})_-]$, and hence
  $\lim_{t \to \infty} \E[( B_{\tau \wedge t})_-] > \E
  [(B_{\tau})_-]$.
  
  



\end{proof}

\section{Diffusions with a given law at an exponential time}

Let us begin by recalling the construction of a diffusion's speed measure in terms of its exponential time law in \cite{MK:12}. 

Given an integrable probability measure $\mu$ on $I$, let
$U^\mu(x)=\int_{I} |x-y| \mu(dy)$, $C^\mu(x)=\int_{I} (y-x)^+ \mu(dy)$
and $P^\mu(x)=\int_{I}(x-y)^+ \mu(dy)$. Let $T_{\lambda}$ be an
exponentially distributed random variable, independent of $B$ with
mean $1/\lambda$. The following theorem summarises the main results in
\cite{MK:12}.

\begin{theorem} \label{t:main}
Let $X=(X_t)_{t \geq 0}$ be a diffusion in natural scale with Wronskian $W_\lambda$. Then $X_{T_{\lambda}} \sim \mu$ if and only if the speed measure of $X$ satisfies
\begin{equation} \label{eq:speeddecomp}
m(dx)= \left\{\begin{array}{ll}
\frac{1}{2\lambda} \frac{\mu(dx)}{P^\mu(x)-P^\mu(x_0)+1/W_\lambda},  &\;   a < x \leq x_0 \\
\frac{1}{2\lambda} \frac{\mu(dx)}{C^\mu(x)-C^\mu(x_0)+1/W_\lambda},  &\; x_0 \leq x < b.  
\end{array}\right.
\end{equation}
\end{theorem}

Since $m$ is positive, $1/W_\lambda \ge \max\{C^\mu(x_0),P^\mu(x_0)\}$. Note that $C^\mu(x_0) \geq (\leq) \ P^\mu(x_0)$ if $x_0 \leq \ (\geq) \ \bar{x}_\mu$. 

The decomposition of the speed measure in Theorem \ref{t:main} is
essentially related to the $\lambda$-potential of a diffusion. Recall
that for a diffusion $X$, the $\lambda$-potential (also known as the
resolvent density) of $X$ is defined as $u_\lambda(x,y)=\E_x
\left[\int_0^\infty e^{- \lambda t} dL_{A_t}^y(t)\right]$. This has
the natural interpretation that $\lambda
u_\lambda(x_0,y)=\E_{x_0}[L_{A_{T_{\lambda}}}^y]$ is the expected
local time of $X$ at $y$ up until the exponentially distributed time
$T_{\lambda}$.

\begin{corollary} \label{c:lpotential}
The $\lambda$-potential of a diffusion $X$ with $X_{T_{\lambda}} \sim \mu$ satisfies
\[
u_\lambda(x_0,y)= \left\{\begin{array}{ll}
2(P^\mu(y)-P^\mu(x_0))+2/W_\lambda,  &\;  a < y \leq x_0 \\
2(C^\mu(y)-C^\mu(x_0))+2/W_\lambda,  &\; x_0 \leq y < b.  
\end{array}\right.
\]
Moreover, $\varphi(x) = u_{\lambda}(x_0,x)$ for $x \le x_0$ and $\phi(x) = u_{\lambda}(x_0,x)$ for $x \ge x_0$.
\end{corollary}

\begin{proof}
It follows from the calculations in \cite{MK:12} (p.~4) that
\[
\E_y[e^{-\lambda H_{x_0}}]= \left\{\begin{array}{ll}
W_\lambda(P^\mu(y)-P^\mu(x_0))+1,  &\;  a< y \leq x_0 \\
W_\lambda(C^\mu(y)-C^\mu(x_0))+1,  &\; x_0 \leq y < b.  
\end{array}\right.
\]
Since $u_\lambda(x_0,y)=\frac{2}{W_\lambda} \E_y[e^{-\lambda H_{x_0}}]$ (cf. Theorem 50.7, V.50 in Rogers and Williams \cite{rogers} for classical diffusions and It\^o and McKean \cite{mckean} for generalised diffusions), the result follows.
\end{proof}

\begin{example} \label{ex:jump} Let $\mu=\frac{1}{3} \delta_{0} +
  \frac{1}{3} \delta_{1/2} + \frac{1}{3} \delta_{1}$. Then
  $P^\mu(x)=\frac{1}{3} x$ for $0 \leq x \leq \frac{1}{2}$ and
  $C^\mu(x)=\frac{1}{3}-\frac{1}{3}x$ for $1/2 \leq x \leq 1$.  Let
  $u_\lambda(1/2,x)$ be the $\lambda$-resolvent of a diffusion started
  at $1/2$ such that $X_{T_{\lambda}} \sim \mu$. Then
\[
u_\lambda(1/2,x)= \left\{\begin{array}{ll}
2(x/3-1/6)+2/W_\lambda,  &\;  0 \leq x \leq 1/2 \\
2(1/6-x/3)+2/W_\lambda,  &\; 1/2 \leq x \leq  1.  
\end{array}\right.
\]
where $W_\lambda \leq 6$. The speed measure of the consistent diffusion charges only the points $0$, $1/2$ and $1$ and is given by
\[
2 \lambda m(\{x\})= \left\{\begin{array}{ll}
\frac{3}{6/W_\lambda-1},  &\;  x=0, \\
\frac{W_\lambda}{2},  &\;  x=1/2, \\
\frac{3}{6/W_\lambda-1},  &\; x=1.  
\end{array}\right.
\]
Consistent diffusions are birth-death processes in the sense of Feller
\cite{feller}. $\frac{d\Gamma_t}{dt} > 0$ whenever $B_t \in \{0,1,2\}$
and the process is `sticky' there. The process $\Gamma_t$ is constant
whenever $B_t \notin \{0,1,2\}$, so that $A_t$ skips over these time
intervals and $X_t=B_{A_t}$ spends no time away from these
points. Note that if $W_\lambda=6$ then $m(\{0\})=m(\{1\})=\infty$ so
that $\Gamma_t=\infty$ for $t$ greater than the first hitting time of
$0$ or $1$ implying that the endpoints are absorbing. If
$W_\lambda<6$, the endpoints are reflecting but sticky.
\end{example}

Let $H_x=\inf\{u \geq 0 :X_u = x\}$ and let $y \in (a,b)$. 

\begin{lemma} \label{l:boundary}
$\int_{a+} (|x|+1) m(dx) = \infty$ $($resp. $\int^{b-} (|x|+1) m(dx) =
\infty)$ if and only if $1/W_\lambda = P^\mu(x_0)$
$($resp. $1/W_\lambda = C^\mu(x_0))$. 
\end{lemma}

\begin{proof}
We prove the second statement, the first follows similarly.

Clearly, if $1/W_\lambda > C^\mu(x_0)$, then $\int^{b-} \frac{(1+|x|) \mu(dx)}{C^{\mu}(x)-C^{\mu}(x_0)+1/W_\lambda} < \infty$. Conversely, suppose that $1/W_\lambda=C^{\mu}(x_0)$. Suppose first that $b=\infty$. Then $\lim_{x \uparrow \infty} \phi_\lambda(x) = \lim_{x \uparrow \infty} u_\lambda(x_0,x) = 0$, since $C^\mu(x) \downarrow 0$ as $x \uparrow \infty$. It follows by Theorem 51.2 in \cite{rogers} (which holds in the generalised diffusion case) that $\int^{\infty} x m(dx) = \infty$. Now suppose instead that $b< \infty$. Observe that $C^\mu(x)$ is a convex function on $I$, so it has left and right derivatives, while its second derivative can be interpreted as a measure; moreover, $C^{\mu}(b) = 0 = (C^\mu)'(b+)$ and $(C^{\mu})''(dx) = \mu(dx)$.

From the convexity and other properties of $C^{\mu}$, we note that $\frac{(x-b)(C^{\mu})'(x-)}{C^{\mu}(x)} \to 1$ as $x \nearrow b$, and also $\frac{C^{\mu}(x)}{x-b} \ge (C^{\mu})'(x-)$ for $x<b$, so $\frac{1}{x-b} \ge \frac{(C^{\mu})'(x-)}{C^{\mu}(x)}$ and $\frac{(C^{\mu})'(x-)}{C^{\mu}(x)} \searrow -\infty$ as $x \nearrow b$. The claim will follow provided we can show $\lim_{v \nearrow b} \int_u^v \frac{(C^\mu)''(dx)}{C^{\mu}(x)} = \infty$ for $u<b$. Using Fubini/integration by parts, and the fact that $(C^{\mu})'(x)$ is increasing and negative, we see that:
  \begin{align*}
    \int_u^{v-} \frac{(C^\mu)''(dx)}{C^{\mu}(x)} & = \frac{(C^{\mu})'(v-)}{C^{\mu}(v)} - \frac{(C^{\mu})'(u-)}{C^{\mu}(u)} + \int_u^v \left(\frac{(C^{\mu})'(x-)}{C^{\mu}(x)}\right)^2 \, dx\\
    & \ge \frac{(C^{\mu})'(v-)}{C^{\mu}(v)} - \frac{(C^{\mu})'(u-)}{C^{\mu}(u)} + (C^{\mu})'(v-)\int_u^v \frac{(C^{\mu})'(x-)}{C^{\mu}(x)^2} \, dx \\
  & \ge \frac{(C^{\mu})'(v-)}{C^{\mu}(v)} - \frac{(C^{\mu})'(u-)}{C^{\mu}(u)} + (C^{\mu})'(v-)\left(\frac{1}{C^{\mu}(u)} - \frac{1}{C^{\mu}(v)}\right) \\
  & \ge \frac{(C^{\mu})'(v-)- (C^{\mu})'(u-)}{C^{\mu}(u)}.
  \end{align*}
  Letting first $v \to b$, the right hand side is equal (in the limit) to $-\frac{(C^{\mu})'(u-)}{C^{\mu}(u)}$, but we observed above that this is unbounded as $u \nearrow b$, and since the whole expression is increasing in $u$, it must be infinite, as required.
\end{proof}

As Lemma \ref{l:boundary} demonstrates, the behaviour of $X$ at the
boundaries is determined by the value of $W_\lambda$. When $I$ is
unbounded, the boundary behaviour determines whether or not $X$ is a
martingale diffusion. Suppose that $a$ is finite, $b=\infty$. Then
Kotani \cite{Kotani} (see also Delbaen and Shirakawa
(\cite{delbaen2002})) show that $X_{t \wedge H_a \wedge H_b}$ is a
martingale if and only if $\int^{\infty-} x m(dx) = \infty$
(i.e. $\infty$ is not an entrance boundary). By Lemma \ref{l:boundary}
this is equivalent to the conditions $x_0 \leq \bar{x}_\mu$ and
$1/W_\lambda=C^\mu(x_0)$ being satisfied. An analogous observation
holds when $b$ is finite and $a$ is infinite.

Theorem \ref{t:main} and Lemma \ref{l:boundary} provide a natural way
of determining boundary properties by inspection of the decomposition
of the speed measure in terms of $\Prob(X_{T_{\lambda}} \in
dx)$. Furthermore, the decomposition gives us a canonical way of
constructing strict local martingales with a given law at a random
time. For instance if $b$ is infinite, we can generate strict local
martingale diffusions by choosing a measure $\mu$ and setting
$W_\lambda < 1/C^\mu(x_0)$.

\begin{example} \label{ex:Bessel}
Let $m(dx)=\frac{dx}{x^4}$ and $I = (0,\infty)$. Suppose $X_0=1$ and $\lambda=2$. Then $\phi(x)=\frac{x\sinh(\frac{1}{x})}{\sinh(1)}$ and $\varphi(x)=x e^{1-1/x}$ are respectively the strictly decreasing and strictly increasing eigenfunctions of the inverse Bessel process of dimension three, $X$, which solves the equation \[\frac{1}{2} \frac{d^2}{dm dx} f = 2 f.\]

We calculate (cf. Equations (3.2) and (3.3) in \cite{MK:12}),
\begin{equation}
\mu(dx)=\Prob(X_{T_{\lambda}} \in dx)= \left\{\begin{array}{ll}
\frac{\sinh(1)}{2 x^3} e^{-1/x}\, dx  &\; 0<x \leq 1,\\
\frac{e^{-1}}{2 x^3}\sinh(1)\, dx  &\; 1 \leq x.  
\end{array}\right.
\end{equation} 
We find $\bar{x}_\mu=1-1/e$. Further, we calculate $P^\mu(x)=\frac{1}{2}\sinh(1) x e^{-1/x}$ for $x \leq 1$ and
$C^\mu(x)=e^{-1}( x \sinh(1/x)-1)$ for $x \geq 1$. Thus

\begin{equation} 
m(dx)=dx/x^4= \left\{\begin{array}{ll}
\frac{\mu(dx)}{P^\mu(x)}  &\; 0<x \leq 1,\\
\frac{\mu(dx)}{C^\mu(x)-(2C^\mu(1)+2P^\mu(1))}  &\; 1 \leq x. 
\end{array}\right.
\end{equation}
It follows that $X$ is a strict local-martingale diffusion. 
\end{example}

In the example above, the strict local martingale property of the process follows from the fact that $C^\mu$ is `shifted up' in the decomposition of the speed measure
(\ref{eq:speeddecomp}). In general, we  expect reflection at the left boundary $X = a$, when we have had to `shift the function $P^\mu$ up', and we expect the process stopped at $t=H_a \wedge H_b$ to be a strict local martingale when we `shift $C^{\mu}$ up'.

\begin{example} \label{ex:Bessel2}
Let us reconsider the diffusion in Example \ref{ex:Bessel}. Let us construct a diffusion $Y$ with the same law as the inverse Bessel process of dimension three at an exponential time, such that $Y_{t \wedge H_0}$ is a martingale (and therefore, such that $Y_0 = \bar{x}_\mu$). As we are shifting the starting point, we must adjust the speed measure between $\bar{x}_\mu$ and the starting point of the diffusion in 
Example \ref{ex:Bessel}.  By Lemma \ref{l:boundary} we know that the speed measure of the martingale diffusion must satisfy
\begin{equation*} 
m(dx)= \left\{\begin{array}{ll}
\frac{\mu(dx)}{P^\mu(x)} &\; 0<x \leq \bar{x}_\mu,\\
\frac{\mu(dx)}{C^\mu(x)} &\; \bar{x}_\mu \leq x, 
\end{array}\right.
\end{equation*}
Recall that $\bar{x}_\mu=1-1/e$. We calculate
\begin{equation*} 
m(dx) = \left\{\begin{array}{ll}
\frac{1}{x^4} \, dx &\; 0<x \leq \bar{x}_\mu,\\
\frac{(e-e^{-1})e^{-1/x}\, dx}{x^3((e-e^{-1}) x e^{-1/x}+2(\bar{x}_\mu-x))}  &\; \bar{x}_\mu<x \leq 1,\\
\frac{\sinh(1/x)\, dx}{x^3(x\sinh(1/x)-1)}  &\; 1 \leq x.
\end{array}\right.
\end{equation*}
\end{example}

\section{The BLJ embedding for Brownian motion}
The  BLJ construction is remarkably general and can be used to embed distributions in general Hunt processes. Our interest, however, lies in the specific case of embedding a law $\mu$ in Brownian Motion. In this setting, the rich structure of the embedding translates into a one-to-one correspondence to the family of diffusions with a given law at an exponential time. 

Suppose that $x_0 \in [a,b]$, $\mu(\{x_0\}) = 0$ and define 
\begin{equation} \label{eq:Vmu}
V^{x_0}_\mu(x)=\int \E_y[L_{H_{x_0}}^x] \mu(dy).
\end{equation}

We will assume from now on that $\mu$ has a finite first moment and that 
$\mu(\{x_0\})=0$. Then $\kappa_0=\sup\{V_\mu(x) ; x \in [a,b]\} < \infty$ (see, for instance, p.~547 in \cite{BLJ}). By the Corollary on p.~540 in \cite{BLJ} it follows that for each $\kappa \geq \kappa_0$ the stopping time 
\begin{equation} \label{eq:BLJstop}
\tau_\kappa=\inf \left\{t>0: \kappa \int \frac{L_t^x \mu(dx)}{\kappa-V^{x_0}_\mu(x)} > L_t^{x_0} \right\}
\end{equation}
embeds $\mu$ in $B^{x_0}$, i.e. $B^{x_0}_{\tau_{\kappa}} \sim \mu$. Moreover, we have $\E_{x_0}[L^{x_0}_{\tau_{\kappa}}]=\kappa$. Finally, the stopping time $\tau_{\kappa_0}$ is optimal in the following sense; if $\sigma$ is another embedding of $\mu$ in $B_{x_0}$, then for every $x \in I$, $\E_{x_0}[L^x_\sigma] \geq \E_{x_0}[L^x_{\tau_{\kappa_0}}]$.  From Lemma~\ref{lem:minimal}, it follows that the stopping time $\tau_{\kappa_0}$ is minimal in the sense of Definition~\ref{def:minimal}.

\section{Embeddings and diffusions}
We will now show that each BLJ embedding $\tau_{\kappa}$ of $\mu$ in
$B$ corresponds to a local-martingale diffusion $X$ such that
$X_{T_{\lambda}} \sim \mu$, for an independent exponentially
distributed random variable $T_{\lambda}$. The correspondence is a
consequence of the role played by the functional $\Gamma_u=\int_{\R}
L_u^x m(dx)$, where $m$ is defined as in (\ref{eq:speeddecomp}) for a
fixed starting point $x_0$ and a target law $\mu$ on $(a,b)$. We have
already seen that $B_{A_{T_\lambda}} \sim \mu$, where
$A_t=\Gamma_t^{-1}$. We will now observe that for $\kappa \geq
\kappa_0$, the BLJ embedding $\tau_\kappa$ introduced above can be
re-written as $\tau_\kappa=\inf\{t \geq 0 : \Gamma_t \lambda \kappa >
L_t^{x_0}\}$. The remarkable role played by the additive functional
$\Gamma$ in both settings, and the resulting correspondence between
embeddings and diffusions, is the subject of this section. In order to
better understand the connection, we first make the following
distinction: let $\tau_{\kappa}^* = \inf\{t \ge 0: \Gamma_t >
T_{\lambda}\}$, where $T_{\lambda}$ is an independent exponential
random variable with mean $1/\lambda$, and $m(dx)$ is given by
\eqref{eq:speeddecomp}, with $\kappa=2/W_{\lambda}$. It follows that
$B_{\tau^*_\kappa} \sim \mu$ for $\kappa \ge
2\max\{C^{\mu}(x_0),P^\mu(x_0)\}$, and (from
Corollary~\ref{c:lpotential}) we have $\E_{x_0}[L^{x_0}_{T_{\lambda}}]
= \kappa$.

The connection between the quantities that define BLJ embeddings and
those that define diffusions follows from relating the
$\lambda$-potential of diffusions, the potential $U^\mu$ of $\mu$ and
$V_\mu^{x_0}$. A number of solutions to the Skorokhod Embedding
Problem, most notably the solutions of Chacon-Walsh
\cite{ChaconWalsh:76}, Az{\'e}ma-Yor (\cite{AzemaYor:79a},
\cite{AzemaYor:79b}) and Perkins \cite{Perkins:86} can be derived
directly from quantities related to the potential $U^\mu$.

\begin{lemma} \label{l:VPo}
\begin{equation} \label{eq:VPo}
V^{x_0}_\mu(x)= \left\{\begin{array}{ll}
2(P^\mu(x_0)-P^\mu(x))  &\; x \leq x_0 \\
2(C^\mu(x_0)-C^\mu(x))  &\; x > x_0.  
\end{array}\right. 
\end{equation}
\end{lemma}

\begin{proof}
Observe that $\E_y[L_{H_{x_0}}^x]=|y-x_0|+|x_0-x|-|y-x|$ is simply the potential kernel for $B^{x_0}$ killed at $x_0$. Integrating we find 
\[V^{x_0}_\mu(x)=\int(|y-x_0|+|x_0-x|-|y-x|)\mu(dy)=U^\mu(x_0)-U^\mu(x)+|x_0-x|.\]
The result follows after some re-arrangement.
\end{proof}

Recalling the definition of $\kappa_0$ in the BLJ embedding, we now observe that $\kappa_0=\max\{2C^\mu(x_0),2P^\mu(x_0)\}$, which is the smallest constant $c$ such that $c-V^{x_0}_\mu(x) \geq 0$ for all $x$.  The situation is illustrated in Figure 1 below.

\begin{figure}[htb]
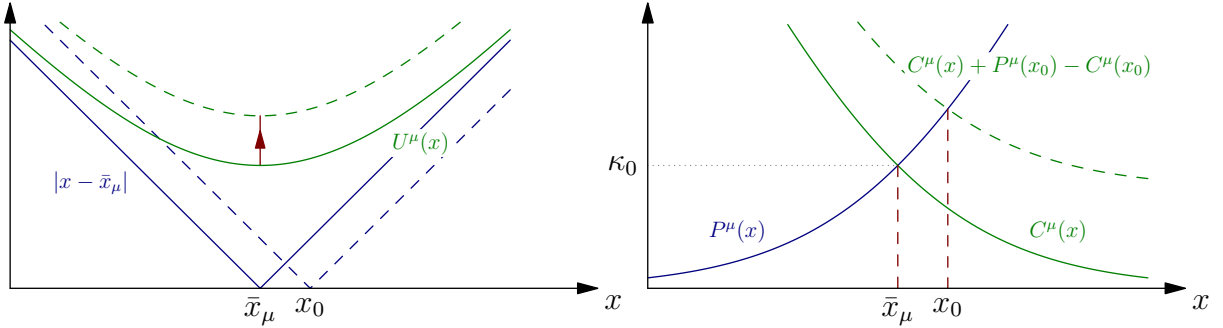

  \centering
  \begin{asy}[width=\textwidth]
    
    import graph;
    
    real xmin = -1;
    real xmax = 1;
    real ymax = 1;

    real u0(real x) {return abs(x);}

    real scale = 1.4;

    real u1(real x) {return (log(2*cosh(scale*x))/scale);}

    // real u0s(real x, real s) {return u0(x-s)

    real skip = 0.35;
    real eps = 0.15;
    real shrink = 0.7;
    
    real yfmax = ymax+eps/2;

    pen q = black+0.5;

    real shift = 0.2;
    real x2 = skip + xmax-xmin+shift;

    real T(real x) {return x-shift;}
    
    real u0s(real x) {return u0(T(x));}
    real u1s(real x) {return u1(x)+shift;}

    bool3 u0smax(real x) { return (u0s(x)<=yfmax);}
    bool3 u1smax(real x) { return (u1s(x)<=yfmax);}

    draw(graph(u0s,xmin,xmax,u0smax),deepblue+0.5+dashed);
    draw(graph(u1s,xmin,xmax,u1smax),deepgreen+0.5+dashed);

    draw((0,u1(0))--(0,u1s(0)),deepred+0.5,MidArrow);

    draw(graph(u0,xmin,xmax),deepblue+0.5);
    label(scale(shrink)*"$|x-{\mmu}|$",(-0.5,u0(-0.5)),SW,deepblue+0.5);
    draw(graph(u1,xmin,xmax),deepgreen+0.5);
    label(scale(shrink)*"$U^{\mu}(x)$",(0.5,u1(0.5)),SE,deepgreen+0.5,filltype=UnFill);

    label("$\mmu$",(0,0),S);
    label("$x_0$",(shift,0),S);

    real Pm(real x) {return u1(x) + x;}
    real Cm(real x) {return u1(x) -x;}

    real Pms(real x) {return Pm(x - x2);}
    real Cms(real x) {return Cm(x - x2);}

    real Cms2(real x) {return Cm(x - x2) + Pm(shift)-Cm(shift);}

    bool3 Pmax(real x) { return (Pms(x)<=yfmax);}
    bool3 Cmax(real x) { return (Cms(x)<=yfmax);}
    bool3 C2max(real x) { return (Cms2(x)<=yfmax);}

    draw(graph(Pms,xmin+x2,xmax+x2,Pmax),deepblue+0.5);
    label(scale(shrink)*"$P^{\mu}(x)$",(x2-0.5,Pms(x2-0.5)),NW,deepblue+0.5);
    draw(graph(Cms,xmin+x2,xmax+x2,Cmax),deepgreen+0.5);
    label(scale(shrink)*"$C^{\mu}(x)$",(x2+0.5,Cms(x2+0.5)),NE,deepgreen+0.5);
    draw(graph(Cms2,xmin+x2,xmax+x2,C2max),deepgreen+0.5+dashed);
    label(scale(shrink)*"$C^{\mu}(x)+P^{\mu}(x_0)-C^{\mu}(x_0)$",(x2,Cms2(x2)),E,deepgreen+0.5,filltype=UnFill);
    
    draw((x2+shift,0)--(x2+shift,Pm(shift)),deepred+0.5+dashed);
    draw((x2,0)--(x2,Pm(0)),deepred+0.5+dashed);
    draw((x2+xmin,Pm(0))--(x2,Pm(0)),gray+0.5+dotted);

    label("$\mmu$",(x2,0),S);
    label("$x_0$",(shift+x2,0),S);
    label("$\kappa_0$",(xmin+x2,Pm(0)),W);

    draw ((xmin,0)--(xmax+shift+eps,0),q,Arrow);
    draw ((xmin,0)--(xmin,ymax+eps),q,Arrow);

    draw ((xmin+x2,0)--(xmax+eps+x2,0),q,Arrow);
    draw ((xmin+x2,0)--(xmin+x2,ymax+eps),q,Arrow);

    label("$x$",(xmax+shift+eps,0),SE);
    label("$x$",(xmax+eps+x2,0),SE);

  \end{asy}
  \caption{If $x_0=\bar{x}_\mu$, then
    $\kappa_0-V_\mu^{{x_0}}(x)=U^\mu(x)-|x-\bar{x}_\mu|$. If $x_0 \neq
    \bar{x}_\mu$, then the potential must be shifted upwards to lie
    above $|x_0-x|$ everywhere. Thus, if $x_0 < \bar{x}_\mu$,
    $\kappa_0-V_\mu^{x_0}(x)=U^\mu(x)+(2C^\mu(x_0)-U^\mu(x_0))-|x_0-x|$,
    while if $x_0>\bar{x}_\mu$,
    $\kappa_0-V_\mu^{x_0}(x)=U^\mu(x)+(2P^\mu(x_0)-U^\mu(x_0))-|x_0-x|$
    (see left picture). Note that if $x \geq x_0>\bar{x}_{\mu}$,
    $\frac{1}{2}(\kappa_0-V_\mu^{x_0}(x))=P^\mu(x_0)-C^\mu(x_0)+C^\mu(x)$
    (see right figure).}
\end{figure}


The following Lemma follows from the Corollary in \cite{BLJ}, page~540.
\begin{lemma} \label{l:lambda}
\begin{equation} \label{eq:lambda}
\E_{x_0}[L^{x_0}_{\tau_{\kappa_0}}]= \left\{\begin{array}{ll}
2C^\mu(x_0)  &\; x_0 \leq \bar{x}_\mu\\
2P^\mu(x_0)  &\; x_0 \geq \bar{x}_\mu.  
\end{array}\right.
\end{equation} 
\end{lemma}

\begin{proof}
By the Corollary in \cite{BLJ}, $\kappa_0=\E_{x_0}[L^{x_0}_{\tau_{\kappa_0}}]$. 
The result now follows from the fact that $\kappa_0=\max\{2C^\mu(x_0),2P^\mu(x_0)\}$.
\end{proof}

In comparison, the Corollary on p. 540 in \cite{BLJ} implies the formula $\E_{x_0}[L^x_{\tau_{\kappa_0}}]=\kappa_0-V^{x_0}_\mu(x)$. It is helpful to 
visualise this relationship pictorially, see Fig. 1. The expected local time at $x$ of Brownian motion started at $x_0$ until the minimal BLJ stopping time is the distance at $x$ between the minimally shifted potential function and the potential kernel.

So far we have focused on the minimal embedding $\tau_{\kappa_0}$ for $\mu$.  The picture for the suboptimal embeddings $\tau_{\kappa}$, $\kappa>\kappa_0$ 
is analogous. The picture simply reflects the additional expected local time accrued between the minimal and the non-minimal stopping time. 
The expected local time of a non-minimal embedding is the 
expression in (\ref{eq:lambda}) plus the positive constant $\kappa-\kappa_0$, i.e. $\E_{x_0}[L^x_{\tau_{\kappa}}]=\E_{x_0}[L^x_{\tau_{\kappa_0}}]+(\kappa-\kappa_0)$. 
The graph of the expected local time (see Fig. 1) is given by
shifting the potential picture $\kappa_0-V^{x_0}_\mu(x)$ upwards by
$\kappa-\kappa_0$. 

It follows from these observations that the local times
$\E_{x_0}[L^x_{\tau_{\kappa}}]$ and $\E_{x_0}[L^x_{\tau^*_{\kappa}}]$
are equal for all $\kappa \ge \kappa_0$ and all $x\in [a,b]$. Let us
now make the correspondence between diffusions (through the stopping
times $\tau_\kappa^*$) and the BLJ embeddings precise.  Fix $x_0 \in
\supp(\mu)$ and let $\kappa \geq \max\{2C^\mu(x_0), 2P^\mu(x_0)\}$.

\begin{proposition} \label{p:correspondence} Fix $\lambda>0$ and
  define a Borel measure $m$ via
  $m(dx)=\frac{1}{\lambda}\frac{\mu(dx)}{\kappa-V_\mu^{x_0}(x)}$.
  Define the functional $\Gamma_u=\int_\R L_u^x m(dx)$. Then:
  \begin{enumerate}
  \item The stopping time $\tau_\kappa=\inf\{u \geq 0: \lambda \kappa
    \Gamma_u > L^{x_0}_u\}$ has $B_{\tau_\kappa} \sim \mu$ and
    $\E_{x_0}[{L_{\tau_\kappa}^x}]=\kappa-V_\mu^{x_0}(x)$.
  \item The stopping time $\tau_\kappa^*=\inf\{u \geq 0: \Gamma_u >
    T_{\lambda}\}$ has $B_{\tau_\kappa^*} \sim \mu$ and
    $\E_{x_0}[{L_{\tau_\kappa}^x}]=\kappa-V_\mu^{x_0}(x)$.
  \item The diffusion with speed measure $m$, defined via
    $X_t=B^{x_0}_{A_t}$, where $A$ is the right-continuous inverse of
    $\Gamma_u$, has Wronskian $W_\lambda=\frac{2}{\kappa}$,
    $\lambda$-potential $u_\lambda(x_0,y)=\kappa-V^{x_0}_\mu(y)$ and
    $X_{T_\lambda} = B_{\tau^*} \sim \mu$.
  \end{enumerate}
\end{proposition}

\begin{proof}
  Since $\mu$ has a finite first moment, $\kappa_0 < \infty$ and
  Hypothesis $1$ in \cite{BLJ} is satisfied. Thus by the Corollary in
  \cite{BLJ}, $\tau_\kappa=\inf \left\{u>0: \kappa \int \frac{L_u^x
      \mu(dx)}{\kappa-V^{x_0}_\mu(x)} > L_u^{x_0} \right\} =\inf\{u
  \geq 0 : \lambda \kappa \Gamma_u > L^{x_0}_u\}$ embeds $\mu$ in
  $B^{x_0}$.  It follows from Lemma \ref{l:VPo} and Corollary
  \ref{c:lpotential} that $u_\lambda(x_0,x)=\kappa-V_\mu^{x_0}(x)$ and
  that
  $\E_{x_0}[L^{x_0}_{A_{T_\lambda}}]=\frac{2}{W_\lambda}=\kappa$. By
  Theorem \ref{t:main}, $X_{T_\lambda} \sim \mu$.
\end{proof}

\begin{remark}
The assumption $\mu(\{x_0\})=0$ made in Section~4 is necessary in the construction of the BLJ embedding in \cite{BLJ}, but not in the construction of diffusions with a given exponential time law in \cite{MK:12}. To make the correspondence independent of this assumption on the target law, a modified version of the BLJ embedding with external randomisation can be constructed.
\end{remark}

With the close connection indicated by
Proposition~\ref{p:correspondence}, it is natural to consider if we
can more explicitly explain the connection between $\tau_\kappa^*$ and
$\tau_\kappa$. To do this, we consider the excursion theory that is
behind the correspondence between diffusions and embeddings. We refer
to Rogers~\cite{rogers1989guided} for background on excursion
theory. Let $U:=\{f \in C(\R^+) : f^{-1}(\R \ \{0\}) = (0, \zeta),$
for some $\zeta >0\}$, be the set of excursions of $B$. Then, by a
well-known result due to It{\^o}, there exists a measure $\Xi$ on $u$
such that for any $A \subseteq U$,
\[\Xi(A):= \# \{\mbox{Excursions of} \ B_t \text{ in } A \ \mbox{between} \ L_t=l \ \mbox{and} \ L_t=r \} \sim \text{Poisson}((r-l)n(A)).\] 

Consider the Brownian motion $B_t$, and $\Gamma_t=\int_\R L_u^x m(dx)$, which is increasing. For a test function $g$, define: $\phi(u):=\int n(df) g(f(\Gamma_t^{-1}(u)) \Indi_{\{\Gamma_t(\zeta(f)) \geq u\}}$, the expected value of $g(B_{\Gamma_u^{-1}})$ averaged over the excursion measure, conditioned on the event that $\Gamma_s \geq u$ on the excursion.

If ${T_\lambda}$ is an independent exponential random variable with parameter $\lambda$ then we compute, using the fact that the process $B$ is recurrent, and the memoryless property of the exponential, that
\begin{equation} \label{eq:firstexcursion}
\E[g(B_{\tau_\kappa^*})] = \E[g(B_{\Gamma_{T_\lambda}^{-1}})]=c \int \phi(u) \lambda e^{-\lambda u} du,
\end{equation}
for some constant $c$.

In the BLJ setting, consider the stopping time $\tau_\kappa=\inf\{u
\geq 0 : \Gamma_u \lambda \kappa > L_t^{x_0}\}$. Observe that
$\Gamma_{L_s^{-1}}$ is an increasing process with independent,
identically distributed increments, and hence $Y_s:=s-\lambda \kappa
\Gamma_{L_s^{-1}}$ is a L{\'e}vy process which goes below $0$ on the
excursion for which $\tau_\kappa$ occurs. In particular, we have
\begin{equation} \label{eq:BLJexcursion}
\E[g(B_{\tau_\kappa})]=\E\left[\int_0^{\tau_\kappa} \phi(L_t^{x_0} - \lambda \kappa \Gamma_t) dL_t^{x_0}\right] = \E\left[\int_0^{L^{-1}_{\tau_\kappa}} \phi(Y_s) ds \right],
\end{equation}
where $L^{-1}_{\tau_\kappa} = \inf\{t \geq 0 : Y_t \leq 0\}$ and the first
equality follows from the Maisonneuve exit system, see \cite{BLJ}, p. 544. 

In addition we observe that $Y_t$ is a L{\'e}vy process and so, using
Lemma 1 in \cite{BLJ}, we can rewrite this formula as
\[\E\left[\int_0^{L^{-1}_{\tau_\kappa}} \phi(Y_s) ds \right] = \int_0^\infty \phi(u) \exp(-\psi^{-1}(0) u)du,\]
where $\psi(x)=\frac{1}{t} \log(\E^{0}[\exp(xY_t)])=x-\log(\E[\exp(-\lambda \kappa \Gamma_{L^{-1}_1})])$.

The behaviour of $\psi$ is determined by properties of $Y_t$. In
particular, it is convex and $\psi^{-1}(0)=\inf\{t : \psi(t) >0\} >0$
if and only if $\psi' (0) <0$ which is true if and only if
$\E[Y_1]<0$, i.e.  $\E^{x_0}[\lambda \kappa \Gamma_t] >
\E^{x_0}[L_t^{x_0}]$. It follows that $\E[Y_1] \leq 0$, since
otherwise with positive probability, $Y_t \nearrow \infty$ and hence
$L_{\tau}^{-1} = \infty$ with positive probability, and hence the same
is true for the BLJ stopping time.

Equating the terms in (\ref{eq:BLJexcursion}) and (\ref{eq:firstexcursion}), we now have an identity involving only $\phi$ and we can choose $\kappa$ so that $\psi^{-1}(0)=\lambda$. Note that the constant $c$ appearing in 
(\ref{eq:firstexcursion}) is equal to $\frac{1}{\lambda}$. Moreover, by the embedding property of $\tau_\kappa$ and the construction of $\Gamma$, both equations are equal to $\int g(x) \mu(dx)$. 

\section{Minimal diffusions}
The purpose of this final section is to define a notion of minimality for diffusions which will be analogous to the notion of minimality for the BLJ Skorokhod embeddings.

Suppose we have fixed a starting point  $x_0$ and a law $\mu$ and we are faced with a (Skorokhod embedding) problem of finding a stopping time $\tau$, such that $B_\tau \sim \mu$, where $B$ is a standard Brownian motion started at $x_0$. Then (and especially when we are interested in questions of optimality) it is most natural to search for stopping times $\tau$ which are minimal.

In the family of BLJ embeddings, the embedding $\tau_{\kappa_0}$ is
minimal, and hence (by Lemma~\ref{lem:minimal}) so too is
$\tau_{\kappa_0}^*$, and the notion of minimal embedding carries over
into a notion of minimality for diffusions in natural scale.

\begin{definition} \label{d:minimality} We say that a time homogeneous diffusion $X$ in natural scale and started at $x_0$ is a $\lambda$-minimal diffusion if, whenever $Y$ is another time-homogeneous diffusion in natural scale, with $X_{T_\lambda} \sim Y_{T_\lambda}$ where ${T_\lambda}$ an independent exponential random variable with mean $\frac{1}{\lambda}$, then $\E_{x_0}[L_{{T_\lambda}}^{x_0}(X)] \le \E_{x_0}[L_{{T_\lambda}}^{x_0}(Y)]$.

We say that $X$ is a minimal diffusion in natural scale if $X$ is $\lambda$-minimal for all $\lambda>0$.
\end{definition}

By Proposition \ref{p:correspondence},  the $\lambda$-minimal local-martingale diffusion corresponds to the minimal BLJ embedding $\tau_{\kappa_0}$. As an analogue of the notion of minimality for the BLJ solution to the Skorokhod Embedding Problem, minimality is a natural probabilistic property. A minimal diffusion has the lowest value of $\E_{x_0}[L^y_{T_\lambda}]$ for all $y>0$ among the diffusions with a given law $\mu$ at time ${T_\lambda}$ started at $x_0$ (this follows from Theorem~\ref{t:main}). In terms of diffusion dynamics, the minimal diffusion moves around the state space slower than all other diffusions with the same exponential time law. 

Minimality is a necessary condition for $X_{t \wedge H_a \wedge H_b}$ to be a martingale. For instance, if $a$ is finite and $b$ is infinite, then $X_{t \wedge H_a \wedge H_b}$ is a martingale if and only if $X$ is minimal and $x_0 \leq \bar{x}_\mu$, i.e. $b$ is not an entrance boundary. When $I$ is a finite interval, minimality is a more natural property than the martingale property; every (stopped) diffusion on a finite interval is a martingale diffusion, but there is only one minimal diffusion for every exponential time law. 

Note also that these definitions extend in an obvious way to diffusions which are natural scale: a diffusion which is not in natural scale is minimal if and only if it is minimal when it is mapped into natural scale. (See Example~\ref{ex:Jacobi}.)

We collect these observations in the following result:
\begin{theorem}
  \label{thm:minimal}
  Let $X$ be a time-homogenous diffusion in natural scale. Then the following are equivalent:
  \begin{enumerate}
  \item $X$ is $\lambda$-minimal;
  \item if $X_{T_\lambda} \sim \mu$ and $W_\lambda$ is the Wronskian of $X$, then:
    \begin{equation*}
      1/W_\lambda=\max\{C^\mu(x_0),P^\mu(x_0)\};
    \end{equation*}
  \item $X$ has at most one entrance boundary;
  \item $X$ is a minimal diffusion. 
  \end{enumerate}
\end{theorem}

\begin{proof}
The equivalence of the first two statements follows from Corollary \ref{c:lpotential}. Next, if $X$ is minimal 
then $\E_a[e^{-\lambda H_{x_0}}]=0$ or $\E_b[e^{-\lambda H_{x_0}}]=0$, and hence $X$ has at most one entrance boundary.
Finally, the properties of the boundary points are independent of the choice of $\lambda$, so if $X$ has at most one entrance boundary, then $X$ is minimal.  
\end{proof}

\begin{example} \label{ex:Jacobi}
  A natural class of non-minimal diffusions is the following class of
  Jacobi diffusions in natural scale. On the domain $[0,b]$
  let \[dX_t=(\alpha-\beta X_t)dt+\sigma \sqrt{X_t(b-X_t)}, \ \ X_0
  \in (0,b),\] where $\frac{2\beta}{\sigma^2}-\frac{2\alpha}{\sigma^2
    b}-1 >0$ and $\frac{2\alpha}{\sigma^2 b}-1 >0$. The eigenfunctions
  $\varphi_\lambda$ and $\phi_\lambda$ can be calculated explicitly as
  hypergeometric functions, see Albanese and Kuznetsov \cite{AK}. It
  is shown in \cite{AK} that $\lim_{x \downarrow 0} \varphi_\lambda(x)
  > 0$ and $\lim_{x \uparrow b} \phi_\lambda(x) > 0$.  Let $s$ be the
  scale function of $X$ and consider $Y=s(X)$ with eigenfunctions
  $\overline{\varphi}_\lambda$ and $\overline{\phi}_\lambda$. Then
  $\lim_{x \downarrow s(0)} \overline{\varphi}_\lambda(x) = \lim_{x
    \downarrow 0} \varphi_\lambda(s^{-1}(x)) > 0$. Similarly $\lim_{x
    \uparrow s(b)} \overline{\phi}(x) > 0$. Hence $Y$ is non-minimal.
\end{example}

\begin{example}
  Let $I=(0,1)$ and let $X=(X_t)_{t \geq 0}$ be a diffusion with $X_0=1/2$ and speed measure $m(dx)=\frac{dx}{x^2(1-x^2)}$ (sometimes known as the Kimura martingale cf. \cite{huillet2}). The increasing and decreasing eigenfunctions are $\varphi(x)=\frac{2x^2}{1-x}$ and $\phi(x)=\frac{2(1-x)^2}{x}$ respectively.  Note that both boundaries are natural, so this diffusion is minimal.

  Non-minimal diffusions with the same exponential time law have increasing/decreasing eigenfunctions $\varphi_\delta(x)=\varphi(x)+\delta$ and $\phi_\delta(x)=\phi(x)+\delta$, for some $\delta > 0$, and with reflection at the endpoints. Accordingly, the consistent speed measures indexed by $\delta$ are given by $m_\delta(dx)=\frac{dx}{\sigma_\delta^2(x)}$, where
\[
\sigma_\delta^2(x)=
\begin{cases}
  \left(\frac{\delta}{2}(1-x)+x^2\right)(1-x)^2 & x < \frac{1}{2}\\
  \left(\frac{\delta}{2}x+(1-x)^2\right)x^2 & x \ge \frac{1}{2}.
\end{cases}
\]
\end{example}

\begin{example}
Consider Brownian motion on $[0,2]$ with instantaneous reflection at the endpoints and initial value $B_0=x_0 \in [1,2)$. Let $\lambda=1/2$, then the $1/2$-eigenfunctions solve $\frac{d^2 f}{dx^2} = f$, so the increasing eigenfunction (up to a constant) is $\varphi(x)=\frac{\cosh(x)}{\cosh(x_0)}$ and 
$\phi(x)=\frac{\cosh(2-x)}{\cosh(2-x_0)}$. Note that this diffusion is non-minimal. To construct the minimal
diffusion with the same marginal law at an exponential time,  let $\eta=\cosh(x_0)^{-1}$ and set $\varphi_\eta(x)=\varphi(x)-\eta=\frac{\cosh(x)-1}{\cosh{x_0}}$, and similarly
$\phi_\eta(x)=\phi(x)-\eta=\frac{\cosh(2-x)}{\cosh(2-x_0)}-\frac{1}{\cosh(x_0)}$. The diffusion co-efficient of 
the minimal diffusion is given by $\sigma^2(x) = \frac{\varphi_\eta(x)}{\varphi''_\eta(x)}=1-\frac{1}{\cosh(x)}$ for 
$x \in [0,x_0]$ and by $\sigma^2(x)=\frac{\phi_\eta(x)}{\phi''_\eta(x)}=1-\frac{\cosh(2-x_0)}{\cosh(2-x)\cosh(x_0)}$ for $x \in [x_0,1]$. 
\end{example}

\bibliography{mindiff}

\end{document}